\documentclass[12pt]{amsart}
\usepackage[utf8]{inputenc}
\usepackage[colorlinks=true,pagebackref,hyperindex,citecolor=Green,linkcolor=Blue]{hyperref}
\usepackage[top=1in, bottom=1in, left=1.1in, right=1.1in]{geometry}
\usepackage{amsmath}
\usepackage{amsfonts}
\usepackage{amssymb}
\usepackage{amsthm}
\usepackage{stmaryrd}
\usepackage[all]{xy}
\usepackage{mathrsfs}
\usepackage{enumitem} 
\usepackage[T1]{fontenc}
\usepackage{color}

\makeatletter
\def\namedlabel#1#2{\begingroup
#2%
\def\@currentlabel{#2}%
\phantomsection\label{#1}\endgroup
}
\makeatother

\theoremstyle{plain} 
\newtheorem{theorem}{Theorem}[section]

\newtheorem{corollary}[theorem]{Corollary}

\newtheorem{lemma}[theorem]{Lemma}
\newtheorem{proposition}[theorem]{Proposition}

\newtheorem{notation}[theorem]{Notation}

\newtheorem{theoremx}{Theorem}

\theoremstyle{definition} 
\newtheorem{definition}[theorem]{Definition}

\newtheorem{remark}[theorem]{Remark}
\numberwithin{equation}{subsection}

\renewcommand{\(}{\left(}
\renewcommand{\)}{\right)}
\newcommand{\NN}{\mathbb{N}}
\newcommand{\RR}{\mathbb{R}}
\newcommand{\ZZ}{\mathbb{Z}}
\newcommand{\QQ}{\mathbb{Q}}
\newcommand{\FF}{\mathbb{F}}

\newcommand{\mm}{\mathfrak{m}}
\newcommand{\nn}{\mathfrak{n}}
\newcommand{\pp}{\mathfrak{p}}

\newcommand{\aaa}{\mathfrak{a}}
\newcommand{\bb}{\mathfrak{b}}
\newcommand{\IN}{\operatorname{in}}

\newcommand{\m}{\mathfrak{m}}

\newcommand{\HOM}{\operatorname{hom}}

\newcommand{\jg}{\mathcal{J}_G}

\newcommand{\FRAC}{\operatorname{frac}}

\makeatletter
\def\@tocline#1#2#3#4#5#6#7{\relax
  \ifnum #1>\c@tocdepth 
  \else
    \par \addpenalty\@secpenalty\addvspace{#2}%
    \begingroup \hyphenpenalty\@M
    \@ifempty{#4}{%
      \@tempdima\csname r@tocindent\number#1\endcsname\relax
    }{%
      \@tempdima#4\relax
    }%
    \parindent\z@ \leftskip#3\relax \advance\leftskip\@tempdima\relax
    \rightskip\@pnumwidth plus4em \parfillskip-\@pnumwidth
    #5\leavevmode\hskip-\@tempdima
      \ifcase #1
       \or\or \hskip 1.9em \or \hskip 2em \else \hskip 3em \fi%
      #6\nobreak\relax
    \dotfill\hbox to\@pnumwidth{\@tocpagenum{#7}}\par
    \nobreak
    \endgroup
  \fi}
\makeatother

\newcommand{\Hom}{\operatorname{Hom}}

\newcommand{\fpt}{\operatorname{fpt}}

\newcommand{\ini}{\operatorname{In}}
\newcommand{\reg}{\operatorname{reg}}
\newcommand{\chara}{\operatorname{char}}

\definecolor{blue-violet}{rgb}{0.54, 0.17, 0.89}
\definecolor{Blue}{rgb}{0.01, 0.28, 1.0}
\definecolor{gGreen}{rgb}{0.2, 0.8, 0.2}
\definecolor{Green}{rgb}{0.04, 0.85, 0.32}

\begin{document}

\title{Gorenstein Binomial Edge Ideals}

\author[René González-Martínez]{René González-Martínez}
\address{Universidad Nacional Autónoma de México, Ciudad de México, México}
\email{jrene@ciencias.unam.mx}

\thanks{{$^1$}The author was partially supported by the CONACyT Grants 433381 and 284598 }

\begin{abstract}
We classify connected graphs $G$ whose binomial edge ideal is Gorenstein. The proof uses methods in prime characteristic.
\end{abstract}
\maketitle

\tableofcontents

\section*{Introduction}
Our main goal is to present homological properties of binomial edge ideals. These ideals are a generalization of determinantal ideals and ideals generated by adjacent 2-minors in a $2\times n$ generic matrix. The binomial edge ideals were introduced by Herzog, Hibi, Hreindóttir, Kahle, and Rauh \cite{HHHKR10}  and by Ohtani \cite{O11} independently and about the same time. 

Let $G$ a simple graph (i.e. $G$ has no loops and multiple edges) on the vertex set $V(G)=[n]:=\{1,\ldots,n\}$ with edge set $E(G)$. Let $S=K[x_1,\ldots,x_n,y_1,\ldots,y_n]$ be the polynomial ring on $2n$ variables over a field $K$. The \emph{binomial edge ideal} $\jg$ of $G$ is 
	\begin{equation*}
		\jg:=(f_{ij}\, |\, \{i,j\}\in E(G)\,i<i),
	\end{equation*}	 
	where $f_{ij}=x_iy_j-x_jy_i$ for $i<j$.
The properties of binomial edge ideals have been studied vastly by many researchers, for instance:
\begin{itemize}
	\item Cohen-Macaulayness \cite{BN17,Bipartite,EHH11,HHHKR10,KM15,MR18,RR14,R13,R17,Z12},
	\item Betti numbers and regularity \cite{B16,CDI16,dAH17,EZ15,KM14,KM16,MM13,KM12,KM17,PZ14},
	\item Gr\"{o}bner bases \cite{BBS16,CR11,HHHKR10,O11}.
\end{itemize}
	Herzog et al. characterized the graphs whose binomial edge ideal has quadratic Gr\"{o}bner base. 
		For a graph $G$, the generators $f_{ij}$ of $\jg$ form a quadratic Gr\"{o}bner basis if and only if for all edges $\{i, j\}$ and $\{k, l\}$ with $i < j$ and $k < l$ one has $\{j, l\}\in E(G)$ if $i = k$, and $\{i, k\}\in E(G)$ if $j = l$ \cite[Theorem 1.1]{HHHKR10}.
	A graph $G$ that satisfies the aforementioned condition is called \emph{closed with respect to the given labelling of the vertices}. We say that a graph $G$ is \emph{closed} if there exists a labeling of its vertices such that $G$ is closed with respect to that labeling.
	
	Ene, Herzog and Hibi proved that if $G$ is a closed graph, then  $S/\jg$ is Gorenstein if and only if $G$ is a path \cite[Corollary 3.4]{EHH11}.	This motivated us to prove the main result of this paper. 
	
	\begin{theoremx}
		Let $G$ be a connected graph such that $S/\jg$ is Gorenstein. Then, $G$ is a path.
	\end{theoremx}

	This is achieved through the applications of methods in prime characteristic, in particular, $F$-pure thresholds \cite{TW04}. Along the way we compute the $F$-pure threshold of binomial edge ideals associated to closed graphs and show that the $F$-pure threshold of the binomial edge ideal coincide with the $F$-pure threshold of the initial ideal of the binomial edge ideal for closed graphs. This follows the line of research which establish that the binomial edge ideal and its initial ideal have similar properties for closed graphs \cite{dAH17, EHH11}.

\section{Background}\label{background}
In this section we recall some notions and known facts regarding binomial edge ideals. In this paper, all graphs are simple. 
\begin{definition}
Let $G$ be a simple graph on $[n]$, and let $i$ and $j$ be two vertices of $G$ with $i<j$.
	A path $i=i_0,i_1,\ldots,i_r=j$ from $i$ to $j$
	is called {\em admissible}, if
	\begin{enumerate}
		\item[(i)] $i_k\neq i_\ell$  for $k\neq \ell$;
		\item[(ii)] for each $k=1,\ldots,r-1$ one has either $i_k<i$ or $i_k>j$;
		\item[(iii)] for any proper subset $\{j_1,\ldots,j_s\}$	of $\{i_1,\ldots,i_{r-1}\}$, the sequence $i,j_1,\ldots,j_s,j$ is not a path.
	\end{enumerate}
	Given an admissible path \[\pi: i=i_0,i_1,\ldots,i_r=j\] from $i$ to $j$, where $i < j$, we associate the monomial
	\begin{equation*}
		u_{\pi}=\(\prod_{i_k>j}x_{i_k}\)	\(\prod_{i_\ell<i}y_{i_\ell}\).
	\end{equation*}
\end{definition}

We now recall a characterization of the Gr\"obner basis for binomial edge ideals. 
	\begin{theorem}[{\cite[Theorem 2.1]{HHHKR10}\label{mainresult}}]
		Let $G$ be a simple graph on $[n]$, and let $<$ be the lexicographic order on $S=K[x_1,\ldots,x_n,y_1,\ldots,y_n]$ induced by $x_1>x_2>\cdots >x_n>y_1>y_2>\cdots >y_n$. Then the set of binomials 
		\begin{equation*}
			\mathcal{G}= \bigcup_{i<j} \,\{\,u_{\pi}f_{ij}\,:\;\text{$\pi$ is an admissible path from $i$ to $j$}\,\}
		\end{equation*}
is a reduced Gr\"obner basis of $\jg$.
	\end{theorem}
Since $\jg$ has a square-free Gr\"obner basis, we conclude that $\jg$ is a radical ideal \cite[Corollary 2.2]{HHHKR10}.	
	
\begin{remark}\label{bin}
	Every monomial $u_{\pi}x_iy_j$ such that $\pi$ is an admissible path from $i$ to $j$ is the initial term of an element of $\mathcal{G}$. From the fact that $\mathcal{G}$ is a reduced Gr\"{o}bner basis, we conclude that the set $\ini(\mathcal{G})$ is the minimal generating set of the ideal $\ini(\jg)$.
\end{remark}	

	The binomial edge ideal of a path $P_n$ with $n$ vertices is a complete intersection having $n-1$ generators of degree 2 and $\reg(S/\mathcal{J}_{P_n}) = n-1$ \cite{EZ15}. The following results state that regularity $n-1$ implies that the graph is a path.
	
	\begin{theorem}[{\cite[Theorem 1.1]{MM13}\label{regbei}}]
		Let $G$ be a graph on the set $[n]$ of vertices. Then 
		\begin{equation*}
			\reg(S/\jg)\leq n-1
		\end{equation*}
	\end{theorem}
	\begin{theorem}[{\cite[Theorem 3.4]{KM16}\label{theo1}}]
		Let $G$ be a graph on $[n]$ which is not a path. Then $\reg(S/\jg)\leq n-2$
	\end{theorem}
	We now recall the Pl{\"u}cker relation for binomials. This plays an important role while we study 
	the $F$-pure threshold for $S/\jg$.
	\begin{proposition}[Pl{\"u}cker relation]
		Let  $i<j<k<l$  be positive integers. Then,
\begin{equation*}
			f_{ij}f_{kl}-f_{ik}f_{jl}+f_{il}f_{jk}=0.
\end{equation*}
	\end{proposition}		
		
\section{F-pure thresholds of graded rings}
	In this section we introduce the basic definitions for methods in prime characteristic. We also list some properties for the $F$-pure threshold of standard graded rings. We consider the $F$-pure threshold with respect to $\mm$, the maximal graded ideal. This invariant was introduced by Takagi and Watanabe \cite{TW04}. The $F$-pure threshold is related to the log-canonical threshold \cite{BFS13, TW04}, and roughly speaking measures the asymptotic splitting order of $\mm$.
	\begin{definition}
		Let $R$ be a Noetherian ring of prime characteristic $p$. We say that $R$ is \textit{$F$-finite} if it is finitely generated $R$-module via the action induced by the Frobenius endomorphism
		\begin{align*}
			F:R&\rightarrow R\\
			r&\mapsto r^p.
		\end{align*}
		For $e\in \NN$, let $F^e:R\rightarrow R$ the $e$-th iteration of the Frobenius endomorphism on $R$. If $R$ is reduced, $R^{1/p^e}$ denotes the ring of $p^e$-th roots of $R$. We often identify $F^e$ with the inclusion $R\subseteq R^{1/p^e}$. In this case,  $R$ is $F$-finite if and only if $R^{1/p}$ is a finitely generated $R$-module. For a standard graded $K$-algebra $(R,\mm,K)$, $R$ is $F$-finite if and only if $K$ is $F$-finite, that is, if and only if $[K:K^p]<\infty$. A ring $R$ is called \textit{$F$-pure} if $F$ is a pure homomorphism of $R$-modules, that is $F\otimes1:R\otimes M\rightarrow R\otimes M$ is injective for all $R$-modules $M$. We say that $R$ is called \textit{$F$-split} if $F$ is a split monomorphism. 
If $R$ is an $F$-finite ring, $R$ is $F$-pure if and only $R$ is $F$-split \cite[Corollary 5.3]{HRFpurity}.
		 Let $J\subseteq R$ an ideal we write $J^{[p]}:=(x^p|\,x\in J)$.
	\end{definition}
	\begin{lemma}[Fedder's Criterion for graded rings {\cite[Theorem 1.12]{F83}\label{fed}}]
		Let $K$ be a field of prime characteristic $p$, and  $S=K[x_1,\ldots,x_n]$  be a polynomial ring over $K$. Let $\mm=(x_1,\ldots,x_n)$ be the irrelevant maximal ideal of $S$, and  $I\subseteq \mm$ be a homogeneous ideal of $S$. Then $S/I$ is $F$-pure if $I^{[p]}:I\nsubseteq \mm^{[p]}$.
	\end{lemma}
	\begin{definition}[{\cite[Definition 2.1]{TW04}}]
		Let $(R,\mm,K)$ be a standard graded $K$-algebra which is $F$-finite and $F$-pure, and let $I\subseteq R$ be a homogeneous ideal. For a real number $\lambda\geq0$, we say that $(R,I^\lambda)$ is \textit{$F$-pure} if for every $e\gg0$, there exists an element $f\in I^{\lfloor (p^e-1)\lambda\rfloor}$ such that the inclusion of $R$-modules $f^{1/p^e}R\subseteq R^{1/p^e}$ splits.
The $F$-pure threshold of $I$ is defined by
			\begin{equation*}
				\fpt(I):=\sup\{\lambda\in \RR_{\geq0}|(R,I^\lambda)\text{ is } F \text{ pure}\}.
			\end{equation*}
			If $I=\mm$, we denote the $F$-pure threshold by $\fpt(R)$.
	\end{definition}
	
	\begin{definition}
		Let $(R,\mm,K)$ be a standard graded $K$-algebra which is $F$-finite and $F$-pure, and $J\subseteq \mm$ be and ideal. We define
			\begin{equation*}
				I_e(R):=\{r\in R|\varphi(r^{1/p^e})\in\mm\text{ for every }\varphi\in\Hom(R^{1/p^e},R)\}
			\end{equation*}
and
		\begin{equation*}
			b_J(p^e):=\max\{r|J^r\nsubseteq I_e(R)\}.
		\end{equation*}
	\end{definition}
	\begin{proposition}[{\cite[Proposition 3.10]{DSNB}}]
		Let $(R,\mm,K)$ be a standard graded $K$-algebra which is $F$-finite and $F$-pure. Let $J\subseteq R$ be a homogeneous ideal. Then
		\begin{equation*}
			\fpt(J)=\lim_{e\to\infty}\frac{b_J(p^e)}{p^e}.
		\end{equation*}
	\end{proposition}
	
	\begin{lemma}[{\cite[Lemma 4.2]{DSNB}\label{mins}}]
		Let $S=K[x_1,\ldots,x_n]$ be a polynomial ring over an $F$-finite field $K$. Let $\nn=(x_1,\ldots,x_n)$ denote the maximal homogeneus ideal. Let $I\subseteq S$ be an homogeneous ideal such that $R:=S/I$ is an $F$-pure ring, and let $\mm=\nn R$. Then, 
		\begin{equation*}
			\min\left\lbrace s\in\NN\middle|\left[\frac{I^{[p^e]}:I+\nn^{[p^e]}}{\nn^{[p^e]}}\right]_s\neq0\right\rbrace=n(p^e-1)-b_\mm(p^e).
		\end{equation*}
	\end{lemma}
	\begin{theorem}[{\cite[Theorem 7.3]{DSNB}\label{prop1}}]
	Let $S=K[x_1,\ldots,x_n]$ be a polynomial ring over an $F$-finite field $K$ of prime characteristic $p$. Let $I$ be a homogeneous ideal such that $R=S/I$ is $F$-pure and Gorenstein. Then, 
		\begin{equation*}
			\reg_S(R)=\dim(R)-\fpt(R).
		\end{equation*}
	\end{theorem}
	\begin{proposition}	
		Let $S=K[x_1,\ldots,x_n,y_1,\ldots,y_m]$ the ring of polynomials in $n+m$ variables over a field $K$, and let $J=\aaa+\bb\subseteq R$ be an ideal such that $\aaa$ is an ideal in the variables $\{x_1,\ldots,x_n\}$ and $\bb$ is an ideal in the variables $\{y_1,\ldots,y_m\}$. Then, 
			\begin{equation*}
				\fpt(S/J)=\fpt(K[x_1,\ldots,x_n]/\aaa)+\fpt(K[y_1,\ldots,y_m]/\bb).
			\end{equation*}
	\end{proposition}
	\begin{proof}
		We note that 
		$$S/J\cong K[x_1,\ldots,x_n]/\aaa\otimes K[y_1,\ldots,y_m]/\bb$$ by properties of tensor product. We now prove that $b_J(p^e)=b_\aaa(p^e)+b_\bb(p^e)$.		
		Set $\gamma=b_J(p^e)$. By the definition of $b_J(p^e)$, we have that $J^\gamma\nsubseteq I_e(S)$. This means that there exists a generator $r$ of $J^\gamma$ and also a morphism $\varphi\in \Hom(S^{1/p^e},S)$ such that $\varphi(r^{1/p^e})\notin\mm$. This generator can be written as $r=a_{i_1}a_{i_2},\ldots,a_{i_s}b_{j_1},\ldots,b_{j_t}$ for some generators $\{a_{i_1}a_{i_2},\ldots,a_{i_s}\}$ of $\aaa$ and some generators $\{b_{j_1},\ldots,b_{j_t}\}$ of $\bb$ with $\gamma=s+j$. This element correspond by 
		\begin{equation}\label{eq1}
			S\cong K[x_1,\ldots,x_n]\otimes		K[y_1,\ldots,y_m]
		\end{equation}		 
		to a element $\alpha\otimes\beta$ in the tensor product with $\alpha\in\aaa^s$ and $\beta\in\bb^t$. We have the next composition of morphisms:
		\begin{alignat*}{7}
			K[x_1,\ldots,x_n]^{1/p^{e}}	&& \longrightarrow 	&&S^{1/p^e}	&& \overset{\varphi}\longrightarrow 	&&& S					&\twoheadrightarrow 	&&&K[x_1,\ldots,x_n] \\
			\alpha^{1/p^e}				&& \longmapsto		&&r^{1/p^e}	&& \longmapsto 		&&& \varphi(r^{1/p^e})	&\mapsto 				&&&\overline{\varphi(r^{1/p^e})},
		\end{alignat*}
		where the leftmost morphism send $x^{1/p^e}\mapsto x^{1/p^e}\otimes \beta^{1/p^e}$ and the rightmost is the natural projection. We have an element $\alpha\in\aaa^s$ and a morphism 
			\begin{equation*}
				\varphi*\in\Hom(K[x_1,\ldots,x_n]^{1/p^{e}},K[x_1,\ldots,x_n])
			\end{equation*}
		(the composition of the morphisms above). Then, $\varphi*$ sends $\alpha^{1/p^e}$ to $\overline{\varphi(r^{1/p^e})}$. We note that $\overline{\varphi(r^{1/p^e})}$ is not in the maximal ideal of $K[x_1,\ldots,x_n]$, because $\varphi(r^{1/p^e})\notin\mm$. 		
		Now we note that 
			\begin{equation*}
				b_\aaa(p^e)=\max(r|\aaa^r\nsubseteq I_e(K[x_1,\ldots,x_n]))\geq s.
			\end{equation*}
			By symmetric argument, we obtain that 
			\begin{equation*}
				b_\bb(p^e)=\max(r|\bb^r\nsubseteq I_e(K[y_1,\ldots,y_m]))\geq t
			\end{equation*}
			proving that $b_J(p^e)\leq b_\aaa(p^e)+b_\bb(p^e)$.			
			For the reverse inequality let $s=b_\aaa(p^e)$, $t=b_\bb(p^e)$, $\alpha\in\aaa^s$, $\beta\in\bb^t$,
			 $$\varphi\in\Hom(K[x_1,\ldots,x_n]^{1/p^e},K[x_1,\ldots,x_n])$$ 
and			 
			 $$\psi\in\Hom(K[y_1,\ldots,y_m]^{1/p^e},K[y_1,\ldots,y_m])$$ such that $\varphi(\alpha^{1/p^e})$ and $\psi(\beta^{1/p^e})$ are not in their respective maximal ideals. Then, by equation \ref{eq1}, $\alpha\otimes\beta$ corresponds to an element $\alpha\beta\in J^{t+s}$. Since
			\begin{equation*}
				(\varphi\otimes\psi)(\alpha^{1/p^e}\otimes\beta^{1/p^e})=\varphi(\alpha^{1/p^e})\psi(\beta^{1/p^e})\notin\mm.
			\end{equation*}
			Then $b_J(p^e)\geq t+s$.
	\end{proof}
	\begin{theorem}[{\cite[Theorem 4.7]{DSNB}\label{comp}}]
		Let $(R,\mm,K)$ a standard graded $K$-algebra wich is $F$-finite and $F$-pure, and let $J\subseteq R$ a compatible ideal. Then,
		\begin{equation*}
			\fpt(R)\leq\fpt(R/J).
		\end{equation*}
		In particular,
		\begin{equation*}
			\fpt(R)\leq\fpt(R/\pp)
		\end{equation*}
		for every minimal prime ideal $\pp$ of $R$.
	\end{theorem}
	
	The next result shows us how to compute the $F$-pure threshold of a squarefree monomial ideal. 
	
	\begin{proposition}\label{fptsfm}
		Let $I$ be a square-free monomial ideal of $S=K[x_1,\ldots,x_n]$. Then $\fpt(S/I)$ is equal to the number of variables that do not appear in its minimal set of generators.
 	\end{proposition}
	\begin{proof}
		Let $\{x_{i_1},\ldots,x_{i_t}\}$ be the set of variables that appear in the minimal generating set of  $I$. Then,
		\begin{equation*}
			x_{i_1}^{p^e-1}\cdots x_{i_t}^{p^e-1}\in(I^{[p^e]}:I)\backslash\mm^{[p^e]}.
		\end{equation*}
		Set
		
		\begin{equation*}
			d:=\deg{x_{i_1}^{p^e-1}\cdots x_{i_t}^{p^e-1}}.
		\end{equation*}
		For every monomial $m$ of degree less than $d$, we have that $m\notin(I^{[p^e]}:I)\backslash\mm^{[p^e]}$. By Lemma \ref{mins} $b_\mm(p^e)=(n-t)(p^e-1)$. Dividing both sides by $p^e$ and taking the limit as $e$ go to infinity yields the desired result.
	\end{proof}

\section{F-pure thresholds of binomial edge ideals}
	Through this section $K$ is a $F-finite$ field of characteristic $p$.
	For a sequence $v_1, \ldots,v_s$ of natural numbers, we set
		\begin{equation*}
			f_{v_1,\ldots,v_s}:=f_{v_1v_2}^{p-1}\cdots f_{v_{n-1}v_n}^{p-1}
		\end{equation*}
	taking $f_{ji}:=-f_{ij}$ for $j>$i. 
	\begin{proposition}\label{f1}
		If $\{a,b\}\in E(G)$. Then,
			\begin{equation*}
				f_{v_1,\ldots,c,a,b,d,\ldots,v_s}\equiv f_{v_1,\ldots,c,b,a,d,\ldots,v_s}\,\mod\jg.
			\end{equation*}
	\end{proposition}
	\begin{proof}
		By the Pl{\"u}cker relations, we have that
			\begin{align*}
				f_{ca}^{p-1}f_{ab}^{p-1}f_{bd}^{p-1}&=f_{ab}^{p-1}(f_{cb}f_{ad}-f_{cd}f_{ab})^{p-1}\\
													&=f_{ab}^{p-1}\sum_{i=0}^{p-1}\binom{p-1}{i}(-1)^i f_{cb}^{p-i-1} f_{ad}^{p-i-1} f_{cd}^if_{ab}^i\\
													&=\sum_{i=0}^{p-1}(-1)^i\binom{p-1}{i} f_{cb}^{p-i-1} f_{ad}^{p-i-1} f_{cd}^i f_{ab}^{p+i-1}.
			\end{align*}			 
		By assumption $f_{ab}\in\jg$, so all the terms of the sum with $i>0$ are contained in $\jg^{[p]}$. This gives
			\begin{align*}
				f_{v_1,\ldots,c,a,b,d,\ldots,v_s}&\equiv f_{v_1v_2}^{p-1}\cdots f_{ca}^{p-1}f_{ab}^{p-1}f_{bd}^{p-1}\cdots f_{v_{n-1}v_n}^{p-1}\\
												&\equiv f_{v_1v_2}^{p-1}\cdots f_{cb}^{p-1} f_{ad}^{p-1} f_{ab}^{p-1}\cdots f_{v_{n-1}v_n}^{p-1}\,\\
												&\equiv f_{v_1,\ldots,c,b,a,d,\ldots,v_s}
			\end{align*}
		modulo $\jg$.
	\end{proof}
	\begin{theorem}\label{sgrad}
		Let $G$ be a simple  connected closed graph which is not the complete graph and let $S:=K[x_1,\ldots,x_n,y_1,\ldots,y_n]$ and $\mm$ be the maximal homogeneous ideal. Then, $S/\jg$ is $F$-pure, and
			\begin{equation*}
				\min\left\lbrace s\in\NN\middle|\left[\frac{\jg^{[p]}:\jg+\mm^{[p]}}{\mm^{[p]}}\right]_s\neq0\right\rbrace\leq 2(n-1)(p-1).
			\end{equation*}
	\end{theorem}
	\begin{proof}
		By the Fedder's criterion (Lemma \ref{fed}), it suffices to show that 
		\begin{equation*}
		f_{1,2,\ldots,n}\in (\jg^{[p]}:\jg)\backslash \mm^{[p]}.
		\end{equation*}
		First we prove that $f_{1,\ldots,n}\notin\mm$. Let $<$ the lexicographic on $S$ induced by $x_1>\cdots>x_n>y_1>\cdots>y_n$. Then 
		\begin{equation*}
			\ini( f_{1,\ldots,n})=x_1^{p-1}\cdots x_{n-1}^{p-1}y_2^{p-1}\cdots y_n^{p-1}\notin\mm^{[p]}.
		\end{equation*}
		
		Next we show that $f_{1,\ldots,n}\in \jg^{[p]}:\jg$. It is enough to show that $f_{1,\ldots,n}f_{ij}\in\jg^{[p]}$ for all $\{i,j\}\in E(G)$.		
		We assume that $\{i,j\}\in E(G)$. If $j=i+1$, then $f_{1,\ldots,n}f_{ij}\in\jg^{[p]}$. We assume that $j>i+1$.		
		If $j\neq n$, then $\{k,j\}\in E(G)$  for all $k\in\{i+1,\ldots,j-1\}$  \cite[Propsition 1.3]{M17}. Hence, by using repeatedly Proposition \ref{f1}, we obtain
		\begin{equation*}
					f_{1,2,\ldots,i,i+1,\ldots,j-1,j,j+1,\ldots,n}\equiv f_{1,2,\ldots,i,j,i+1,\ldots,j-1,j+1,\ldots,n}\,\mod\jg^{[p]}.
		\end{equation*}
		Since $f_{i,j}^{p-1}$ is a factor of the last expression, we have that $f_{1,\ldots,n}\in \jg^{[p]}:\jg$.		
		
		 If $j=n$, then $i\neq 1$. For if $\{1,n\}\in E(G)$, then $G$ is complete \cite[Proposition 1.3]{M17}.
		By iterating Proposition \ref{f1},
		\begin{equation*}
			f_{1,2,\ldots,i,i+1,\ldots,j-1,j,j+1,\ldots,n}\equiv f_{1,2,\ldots,i-1,i+1,\ldots,i,n}\,\mod\jg^{[p]}.
		\end{equation*}
		Then $f_{in}^{p-1}$ is a factor of the last expression, and $f_{1,\ldots,n}\in \jg^{[p]}:\jg$.
	\end{proof}
	\begin{corollary}
		Let $G$ be a closed graph, let $S:=K[x_1,\ldots,x_n,y_1,\ldots,y_n]$ and let $\mm$ be the maximal homogeneous ideal. Then, $\fpt(S/\jg)=2$.
	\end{corollary}
	\begin{proof}
		If $G$ is complete, then $S/\jg$ is determinantal, and so, $\fpt(S/\jg)=2$ \cite[Proposition 4.3]{STV}.
		First we prove by induction on $e$ that if $f^{p-1}\in(I^{[p]}:I)\backslash\mm^{[p]}$, then $f^{p^e-1}\in(I^{[p^e]}:I)\backslash\mm^{[p^e]}$. 
		
		The step base follows from our assumptions.
		
		For $f^{p^e-1}\notin\mm^{[p^e]}$ we have that 
		\begin{align*}
			(\mm^{[p^e]}:f^{p^e-1})\subseteq\mm&\Rightarrow(\mm^{[p^{e+1}]}:f^{p^{e+1}-p})\subseteq\mm^{[p]}\\
			&\Rightarrow((\mm^{[p^{e+1}]}:f^{p^{e+1}-p}):f^{p-1})\subseteq(\mm^{[p]}:f^{p-1})\subseteq\mm\\
			&\Rightarrow(\mm^{[p^{e+1}]}:f^{p^{e+1}-1})\subseteq\mm.
		\end{align*}
		This means that $f^{p^{e+1}-1}\notin\mm^{[p^{e+1}]}$. 
		
		If $f^{p^e-1}\in I^{[p^e]}:I$, Then,
		\begin{align*}
			f^{p-1}I\subseteq I^{[p]}&\Rightarrow(f^{p-1}I)^{[p^e]}\subseteq I^{[p^{e+1}]}\\
			&\Rightarrow f^{p^{e+1}-p^e}I^{[p^e]}\subseteq I^{[p^{e+1}]}\\
			&\Rightarrow f^{p^{e+1}-p^e}(f^{p^e-1}I)\subseteq f^{p^{e+1}-p^e}(I^{[p^e]})\subseteq I^{[p^{e+1}]}\\
			&\Rightarrow f^{p^{e+1}-1}I\subseteq I^{[p^{e+1}]},
		\end{align*}
		thus $f^{p^{e+1}-1}\in I^{[p^{e+1}]}:I$.		
		This means that $(x_1x_2\cdots x_{n-1}y_2\cdots y_n)^{p^e-1}\in(\jg^{[p^e]}:\jg)\backslash \mm^{[p^e]}$		
		Using Lemma \ref{mins}, we deduce that 
			\begin{align*}
				2n(p^e-1)-b_\mm(p^e)&\leq 2(n-1)(p^e-1)\\
				-b_\mm(p^e)&\leq -2(p^e-1)\\
				\frac{b_\mm(p^e)}{p^e}&\geq \frac{2(p^e-1)}{p^e}\\
				\fpt(S/\jg)=\lim_{e\to\infty}\frac{b_\mm(p^e)}{p^e}&\geq \lim_{e\to\infty}\frac{2(p^e-1)}{p^e}=2.
			\end{align*}
			Since $\mathcal{J}_{K_n}$ is a minimal prime over $\jg$, the reverse inequality is a consequence of \ref{comp} and the fact that $\fpt{\mathcal{J}_{k_n}}=2$.
	\end{proof}
	
 	\begin{proposition}\label{fptini}
 		Let $G$ be a connected graph on $[n]$. Then, $x_n$ and $y_1$ are the only variables that do not appear in the minimal generating set of $\ini(\jg)$. 
 	\end{proposition}
 	\begin{proof}
 		The set of monomials
 		\begin{equation*}
 			\mathcal{H}=\bigcup_{i<j}\{u_\pi x_iy_j\,:\,\pi \textit{ is an admisible path from $i$ to $j$ }\}
 		\end{equation*}
 		generates $\ini(\jg)$ by Remark \ref{bin}. In every of this monomials, the variables $x_n$ and $y_1$ do not appear. We show now that they are the only variables with this property. Let $i\in\{1,\ldots,n-1\}$ and let $\pi:i=i_0,i_1,\ldots,i_{r-1},i_r=i+1$ be a path of minimal length from $i$ to $i+1$, then for each $k=i_1,\ldots,i_{r-1}$, $i_k<i$ or $i_k>i+1$, and by the minimality of $\pi$, for any proper subset $\{j_1,\ldots,j_s\}$ of $\{i_1,\ldots,i_{r-1}\}$, $i=i_0,j_1,\ldots,j_s,i_r=i+1$ is not a path. Hence $\pi$ is admisible, so $x_i$ and $x_i+1$ appear in $\mathcal{H}$ for $i\in\{1,\ldots,n-1\}$.
 	\end{proof}
 	\begin{corollary}
 		Let $G$ be a connected graph on $[n]$, then $\fpt(R/\ini(\jg))=2$.
 	\end{corollary}
 	\begin{proof}
 		This is follows from Theorem \ref{fptsfm} and Proposition \ref{fptini}.
 	\end{proof}
	\begin{remark}
		Let $G$ be closed graph and $S:=K[x_1,\ldots,x_n,y_1,\ldots,y_n]$. Then,
			\begin{equation*}
				\fpt(S/\jg)= \fpt(S/\ini(\jg)).
			\end{equation*}
	\end{remark}
\section{Gorenstein Binomial Edge Ideals} \label{GBEA}

In this section we prove our main result. We start with a preparation theorem regarding $F$-injectivity of square Gr\"{o}bner deformations. We first need to introduce notation.

\begin{notation}\label{Notation}
Let $S=K[x_1,\ldots,x_n]$ be a polynomial ring over a field with maximal homogeneous ideal $\m$.
Let $I$ be an ideal and $<$ a monomial order such that $\IN(I)$ is square-free.
There exists a vector $w\in\NN^n$ such that $\IN_<(I)=\IN_w(I)$ \cite[Proposition 1.11]{S95}.
Let $A=K[t]$ be a polynomial ring, $L=\FRAC(A)$, and $T=A\otimes_K S$.
We set $J=\HOM_w(I)\subseteq T$ the homogenization of $I$, $R=T/J$, and  $\overline{R}=R/xR.$
\end{notation}

\begin{remark}\label{RemNotation}
Under Notation  \ref{Notation}, it is well known that
\begin{enumerate}
\item $A\to R$ is flat;
\item $R/tR =S/\IN_<(I)$;
\item $R/(t-a)R =S/I$ for every $a\in K\setminus \{0\}$;
\item $R\otimes_A L =S/I\otimes_K L$;
\end{enumerate}
\end{remark}

The following result was obtained independently and simultaneously by Varbaro and  Koley \cite{VK}.

\begin{theorem}\label{ThmFinj}
Let $S=K[x_1,\ldots,x_n]$ be a polynomial ring over a field, $K$, of prime characteristic.
Let $I$ be an ideal and $<$ a monomial order such that $\IN_<(I)$ is square-free.
Then, $S/I$ is $F$-injective.
\end{theorem}
\begin{proof}
We use Notation \ref{Notation} and the facts in Remark \ref{RemNotation} in this proof.
We have that $R/tR$ is an Stanley-Reisner ring, and so, $F$-pure. Since $F$-pure rings are $F$-full \cite[Lemma 2.5]{SW} and $F$-injective, $R/tR$ satisfies these properties. Since $t$ is a nonzero divisor, we have that $R$ is also $F$-full and $F$-injective \cite[Theorem 1.1]{MaQuy}. Then, $R\otimes_A L =S/I\otimes_K L$ are both $F$-injective and $F$-full, because these properties are preserved under localization.
We note that $S/I$ is a direct summand of $S/I\otimes_K L$. 
We note that $\m$ expands the maximal homogeneous ideal in $S/I\otimes_K L$. Then,
we have a commutative diagram

 \centerline{
\xymatrix{
H^i_\m(S/I) \ar[d]^{F_{S/I}} \ar[r]^\alpha & H^i_\m(S/I\otimes_K L)\ar[d]^{F_{S/I\otimes_K L}}\\
H^i_\m(S/I) \ar[r]^\alpha          & H^i_\m(S/I\otimes_K L)},
}
\noindent where $\alpha$ denotes the maps induced by the inclusion $S/I\to S/I\otimes_K L$.
Since the horizontal maps split, they are injective. Since $S/I\otimes_K L$ is $F$-injective, we have that 
$F_{S/I\otimes_K L}\circ \alpha =\alpha\circ F_{S/I}$ is injective. Hence, $F_{S/I}$ is injective, and the result follows.
\end{proof}

We are now ready to show our main result in prime characteristic.

	\begin{theorem}\label{gorp}
		Let $S:=K[x_1,\ldots,x_n,y_1,\ldots,y_n]$. Suppose that $\chara(K)=p>0$. Let $G$ be a connected graph such that $S/\jg$ is Gorenstein. Then, $G$ is a path.
	\end{theorem}
	\begin{proof}
By Theorem \ref{ThmFinj}, we have that  $S/\jg$ is $F$-injective. Since $S/\jg$ is Gorenstein, we have that $S/\jg$ is $F$-pure \cite[Lemma 3.3]{F83}.

		Since $G$ is connected, $\mathcal{J}_{k_n}$ is a minimal prime over $\jg$ \cite{HHHKR10} and its dimension is $n+1$. Then,
		\begin{equation}
			\reg(R/\jg)=\dim(R/\jg)-\fpt(R/\jg)\geq(n+1)-2=n-1
		\end{equation}
		where the inequality comes from the fact that if $I\subseteq J$ then $\fpt(\jg)\leq\fpt(\mathcal{J}_{K_n})=2$.
		Hence, $G$ is a path by Theorem \ref{theo1}.

	\end{proof}

In the previous result we estimate the regularity of $R/\jg$ using $F$-pure thresholds. We point out that
the extremal Betti numbers of $R/\jg$ and $R/\IN(\jg)$ coincide, in particular, 
 $\reg(R/\jg)=\reg(R/\IN(\jg))$ \cite[Corollary 2.7]{CV18}.

	We are now ready to prove the main result in this manuscript in  characteristic zero.
	
	\begin{theorem}
		Let $S:=K[x_1,\ldots,x_n,y_1,\ldots,y_n]$. Suppose that $\chara(K)=0$. Let $G$ be a connected graph such that $S/\jg$ is Gorenstein. Then, $G$ is a path.
	\end{theorem}
	\begin{proof}
		Since field extensions do not affect whether a ring is Gorenstein, without loss of generality we can assume that $K=\QQ$.
		
	Let $A=\ZZ[x_1,\ldots,x_n,y_1,\ldots,y_n]$ and
		\begin{equation*}
			J=(x_iy_j-x_jy_i:\{i,j\}\in G\text{ and }i<j)A.
		\end{equation*}
		
	Then, 
		\begin{equation*}
			\reg_S(S/\jg)=\reg_{A\otimes_\ZZ \QQ}(A\otimes_\ZZ \QQ/J\otimes_\ZZ \QQ)=\reg_{A\otimes_\ZZ \FF_p}(A/J\otimes_\ZZ \FF_p)
		\end{equation*}
			and $A/J\otimes_\ZZ \FF_p$ is Gorenstein for $p\gg0$ \cite[Theorem 2.3.5]{HH}. Then,  $	\reg_S(S/\jg)\geq n-1$ from the proof of Theorem \ref{gorp}.
	Hence, $G$ is a path by Theorem \ref{theo1}.
	\end{proof}

\section*{Acknowledgments}
	I  thank Prof. Luis N\'uñez-Betancourt and Prof. Martha Takane for useful advice and suggestions on this project. I  thank Prof. Aldo Conca and Prof. Matteo Varbaro for pointing out a mistake in an earlier version of this manuscript. I thank Prof. Alessandro Di Stefani for helpful discussions. Finally, I thank the anonymous  referee for helpful comments, and constructive remarks on this manuscript.

\bibliographystyle{alpha}

\bibliography{Refs}

\end{document}